\pdfoutput=1
\documentclass[final,fleqn,reqno,11pt]{article}
\usepackage[T1]{fontenc}
\usepackage{etex}
\usepackage[a4paper]{geometry}
\usepackage{lmodern}
\usepackage{mathrsfs}
\usepackage{microtype}
\usepackage{paragraphs}
\usepackage{enumitem}
\usepackage{amsmath}
\usepackage{amsfonts}
\usepackage{amssymb}
\usepackage{mathtools}

\usepackage{tikz}
\usetikzlibrary{decorations.markings}
\usetikzlibrary{decorations.pathmorphing}
\usetikzlibrary{intersections}
\usetikzlibrary{arrows.meta}
\usetikzlibrary{cd}

\usepackage{leading}
\leading{14.5pt}

\makeatletter
\AtBeginDocument{%
  \let\[\@undefined
  \DeclareRobustCommand{\[}{\begin{equation}}%
  \let\]\@undefined
  \DeclareRobustCommand{\]}{\end{equation}}%
}
\makeatother
\mathtoolsset{showonlyrefs,showmanualtags}

\usepackage{hyperref}
\hypersetup{
  colorlinks, 
  linkcolor=-red!75!green!50, 
  citecolor=-red!75!green!50,
  urlcolor=green!30!black
}

\usepackage[initials,nobysame,msc-links]{amsrefs}

\usepackage{amsthm}
\usepackage{thmtools}
\declaretheoremstyle[
        headformat=\NUMBER.\hskip.5ex\NAME\NOTE,
        spaceabove=\paraskip, 
        headfont=\bfseries,
        bodyfont=\itshape
        ]{mystyle}
\declaretheoremstyle[
        numbered=no, 
        spaceabove=\paraskip, 
        bodyfont=\itshape
        ]{mystyleempty}
\declaretheorem[style=mystyleempty, name=Lemma]{Lemma*}
\newcounter{theorem}

\declaretheorem[sibling=theorem, name=Theorem]{Theorem}
\declaretheorem[style=mystyleempty, name=Corollary]{Corollary*}

\makeatletter
\renewenvironment{proof}[1][\proofname]{\par
  \pushQED{\qed}%
  \normalfont \topsep0pt\relax
  \trivlist
  \item[\hskip\labelsep
        \itshape
    #1\@addpunct{.}]\ignorespaces
}{%
  \popQED\endtrivlist\@endpefalse
}
\makeatother

\newlist{thmlist}{enumerate}{1}
\setlist[thmlist]{
        nolistsep,
        ref={\mdseries\textup{(\emph{\roman*})}},
        label={\mdseries\textup{(\emph{\roman*})}},
        }

\newcommand\claim[2][.8]{%
  \begin{minipage}{#1\displaywidth}%
  \itshape
  #2
  \end{minipage}%
}

\DeclareFontFamily{OT1}{slmss}{}
\DeclareFontShape{OT1}{slmss}{m}{n}
     {<-8.5> s*[1.1] rm-lmss8
      <8.5-9.5> s*[1.1] rm-lmss9
      <9.5-11> s*[1.1] rm-lmss10
      <11-15.5> s*[1.1] rm-lmss12
      <15.5-> s*[1.1] rm-lmss17
     }{}
\DeclareSymbolFont{sfoperators}{OT1}{slmss}{m}{n}
\DeclareSymbolFontAlphabet{\mathsf}{sfoperators}
\makeatletter
\def\operator@font{\mathgroup\symsfoperators}
\makeatother

\newcommand\NN{\mathbb{N}}
\newcommand\CC{\mathbb{C}}
\renewcommand\epsilon{\varepsilon}
\newcommand\id{\mathrm{id}}

\DeclarePairedDelimiter\abs{\lvert}{\rvert}
\DeclarePairedDelimiter\lin{\langle}{\rangle}
\DeclareMathOperator{\GL}{GL}
\DeclareMathOperator{\Ext}{Ext}
\DeclareMathOperator{\Tor}{Tor}
\let\nbd\nobreakdash
\newcommand\op{\mathsf{op}}
\newcommand\HH{H\!H}
\newcommand\kk{\Bbbk}
\DeclareMathOperator\Br{Br}
\newcommand\place{\mathord-}
\newcommand\p[2][]{_{#1(#2)}}

\DeclareMathOperator\rad{rad}
\newcommand\jump{\zeta}
\newcommand\ad{\mathsf{ad}}
\renewcommand\dbar[1]{\Bar{\Bar{#1}}}

\setitemize{nolistsep, listparindent=\parindent}

\definecolor{newterm-color}{RGB}{0, 51, 153}
\newcommand\newterm[1]{%
  \textcolor{newterm-color}{\bfseries\itshape #1}%
}

\title{A little bit of extra functoriality for $\Ext$ and the computation
of the Gerstenhaber bracket}

\author{Mariano Su\'arez-\'Alvarez%
  \thanks{This work has been partially supported by the projects UBACYT
  20020130100533BA, PIP-CONICET
  112-201101-00617, PICT 2011-1510 and MATHAMSUD-REPHOMOL. The author
  is a research member
  of CONICET (Argentina).}}

\date{April 18, 2016}

\AtEndDocument{\vspace*{\stretch{1}}{\footnotesize\parindent0pt%
  \textsc{Departamento de Matem\'atica, Facultad de Ciencias Exactas y
  Naturales, Universidad de Buenos Aires. Ciudad Universitaria,
  Pabell\'on I (1428) Ciudad de Buenos Aires - Argentina} \par  
  \textit{E-mail address}: \texttt{mariano@dm.uba.ar} \par
}}

\begin{document}

\maketitle

\begin{abstract}
We show that the action of the Lie algebra $\HH^1(A)$ of outer derivations
of an associative algebra~$A$ on the Hochschild cohomology $\HH^\bullet(A)$
of~$A$ given by the Gerstenhaber bracket can be computed in terms of an
arbitrary projective resolution of~$A$ as an $A$-bimodule, without having
recourse to comparison maps between the resolution and the bar resolution.
\end{abstract}

\vspace*{1cm}

In his classic paper \emph{On the cohomology structure of an associative
ring} \cite{G}, Murray Gerstenhaber introduced a Lie algebra structure on
the Hochschild cohomology $\HH^\bullet(A)$ of an associative algebra~$A$.
This structure played a role in the proof contained in that paper of the
commutativity of the cup product of $\HH^\bullet(A)$, he himself showed
later in~\cite{G:Deformations} that it is related to the deformation theory
of~$A$, and it has ever since been regarded as an important piece of the
cohomological structure of the algebra. There has been a significant amount
of effort expended by many authors in order to study this structure,
specially in recent times.

This Lie algebra structure on~$\HH^\bullet(A)$ is defined in terms of a
particular realization of Hochschild cohomology: the algebra~$A$ has a
canonical bimodule \emph{bar} resolution~$B(A)_\bullet$, the Hochschild
cohomology~$\HH^\bullet(A)$ is canonically isomorphic to the cohomology of
the complex $\hom_{A^e}(B(A)_\bullet,A)$, and the Lie bracket
of~$\HH^\bullet(A)$ is constructed using certain explicit formulas in terms
of cochains in this complex. While this is convenient for many purposes, it
is quite inconvenient in one important respect: we never \emph{compute}
Hochschild cohomology using the bar resolution. In practice, we pick a
projective resolution~$P_\bullet$ of~$A$ which is better adapted to the
task and compute instead the cohomology of the complex
$\hom_{A^e}(P_\bullet,A)$, which is ---thanks to the yoga of homological
algebra--- canonically isomorphic to that of the complex
$\hom_{A^e}(B(A)_\bullet,A)$. In principle, we can transport the Lie
structure on the complex $\hom_{A^e}(B(A)_\bullet,A)$ to
$\hom_{A^e}(P_\bullet,A)$ using those canonical isomorphisms, but actually
doing this depends crucially on having explicit comparison morphisms
$B(A)_\bullet\rightleftarrows P_\bullet$ between the two resolutions
involved. The problem resides in that making such morphisms explicit is
notoriously difficult.

In what follows, we present an approach which allows us to compute in this
situation \emph{part} of the Gerstenhaber Lie bracket on the cohomology of
the complex $\hom_{A^e}(P_\bullet,A)$ \emph{without} having recourse to
comparison morphisms. More precisely, it gives a way to compute the
restriction of the Lie bracket to $\HH^1(A)\times\HH^\bullet(A)$ or, in
other words, the Lie action of the Lie algebra $\HH^1(A)$ of outer
derivations of~$A$ on $\HH^\bullet(A)$. This requires some amount of
lifting of maps to resolutions, as it should be expected, but involving
only the resolution~$P_\bullet$. Carrying this out in concrete examples
seems to be quite feasible.

\bigskip

Let us explain the idea and, at the same time, describe the contents of the
paper. Suppose that $\delta:A\to A$ is a derivation of the algebra~$A$ and
that $M$ is a left $A$-module. We say that a linear map $f:M\to M$ is a
\newterm{$\delta$-operator} on~$M$ if $f(am) = af(m)+\delta(a)m$ for all
$a\in A$ and $m\in M $. From such an~$f$ we construct in
Section~\ref{sect:bit} a linear map
  \[
  \nabla_f:\Ext_A^\bullet(M,M)\to\Ext_A^\bullet(M,M)
  \]
as follows: we pick a projective resolution $P_\bullet$ of~$M$, show that
there exists a morphism of complexes of \emph{vector spaces}
$f_\bullet:P_\bullet\to P_\bullet$ lifting $f:M\to M$ such that each
component $f_i:P_i\to P_i$ is a $\delta$-operator, and then define a
morphism of complexes
$f_\bullet^\sharp:\hom_A(P_\bullet,M)\to\hom_A(P_\bullet,M)$ such that
$f_i^\sharp(\phi)(p)=f(\phi(p))-\phi(f_i(p))$ for each
$\phi\in\hom_A(P_i,M)$ and each $p\in P_i$. The map $\nabla_f$ is the one
induced on homology by~$f^\sharp_\bullet$, and the key point here is that
it depends only on~$\delta$ and~$f$ and not on the choices made. We view
this as exhibiting a little bit of `extra' functoriality on the $\Ext$
functors, now with respect to $\delta$-operators, and find it somewhat
surprising.

Next, in section~\ref{sect:hh} we specialize this to the following
situation. We start with a derivation $\delta:A\to A$, we consider the
derivation $\delta^e=\delta\otimes1+1\otimes\delta:A^e\to A^e$ on the
enveloping algebra~$A^e$ of~$A$, and observe that the map $\delta:A\to A$
is then a $\delta^e$-operator on~$A$ viewed as a left $A^e$-module as
usual. Recalling that the Hochschild cohomology $\HH^\bullet(A)$ can often
be identified with $\Ext_{A^e}^\bullet(A,A)$, our construction then
produces a map
  \[
  \nabla_\delta:\HH^\bullet(A)\to\HH^\bullet(A)
  \]
which can be computed as described above from any $A^e$-projective
resolution of~$A$ endowed with a lifting of~$\delta$. In particular, we can
use the bar resolution to do this: on it there exists a certain canonical
lifting of~$\delta$ and it turns out that the explicit formulas associated
to it for the map $\nabla_\delta$ are precisely the same ones used by
Gerstenhaber to define the map 
  \[
  [\delta,\place]:\HH^\bullet(A)\to\HH^\bullet(A).
  \]
Of course, this means that $\nabla_\delta=[\delta,\place]$ and shows that
we can compute the restriction of the bracket to
$\HH^1(A)\times\HH^\bullet(A)$ using our favorite resolution, which is what
we wanted.

In Section~\ref{sect:examples} we present this computation of the
Gerstenhaber bracket in two ``real life'' examples: truncated path algebras
and crossed products of symmetric algebras $S(V)$ by a finite group~$G$
acting linearly. In the two cases ---and after a certain amount of work
needed to be able to describe explicitly the cohomology itself--- we are
able to exhibit formulas for the bracket. Finally, in the last section,
Section~\ref{sect:tor}, we rapidly explain how a procedure similar to the
one sketched above applies to $\Tor$ functors and, in particular, to the
action of the Lie algebra $\HH^1(A)$ on the Hochschild homology
$\HH_\bullet(A)$.

\bigskip

The very natural problem which we partially solve in this paper, that of
finding a way to compute the Gerstenhaber bracket on Hochschild cohomology
in term of an arbitrary projective resolution, was posed originally by
Gerstenhaber and Samuel Schack in their survey \cite{GS} in 1988. It is
generally agreed that solving it will
require a different perspective on the construction of the
bracket. Ten years later, Stefan Schwede gave in~\cite{Schwede} a beautiful
interpretation of the bracket in terms of actual commutators of paths in
the geometric realization of the nerve of a category of Yoneda extensions
first considered by Vladimir Retakh in~\cite{Retakh} --- it does not
appear, though, that this interpretation leads to a computational device in
practice. The first concrete step forward occurred very recently: in their
preprint \cite{NS}, based on the thesis~\cite{Negron} of the first author,
Cris Negron and Sarah Witherspoon describe an alternate approach to the
computation of the bracket which, under certain conditions ---satisfied,
for example, if the algebra is Koszul--- allows for the computation of the
bracket in terms of a projective resolution. This approach gives a
computation of the `whole' bracket, but has the disadvantage of being very
close in practice to the construction of comparison morphisms, which we
want to avoid.

\bigskip

In what follows we fix a commutative ring~$\kk$ to play the role of ring of
scalars. Throughout $A$ will denote a projective $\kk$-algebra, unadorned
$\otimes$ and $\hom$ will denote $\hom_\kk$ and $\otimes_\kk$, and linear
will mean $\kk$-linear. In particular, the Hochschild
cohomology~$\HH^\bullet(A)$ of~$A$ as a $\kk$-algebra can and will be
identified canonically with the Yoneda algebra $\Ext_{A^e}^\bullet(A,A)$
of~$A$ viewed as a left $A^e$-module, and likewise for homology.

\section{A little bit of extra functoriality for
\texorpdfstring{$\Ext$}{Ext}}
\label{sect:bit}

\paragraph Let us fix an algebra $A$ and a derivation $\delta:A\to A$.
If $M$ is a left $A$-module,
a \newterm{$\delta$-operator} on~$M$ is a linear map $f:M\to M$ such that
for all $a\in A$ and all $m\in M$ we have
  \[
  f(am) = \delta(a)m + af(m).
  \]
While $\delta$-operators are in general not morphisms of $A$-modules, we
have the following:

\begin{Lemma*}
If $M$ is a left $A$-module and $f$,~$f':M\to M$ are $\delta$-operators
on~$M$, then $f-f':M\to M$ is a morphism of $A$-modules.
\end{Lemma*}

\begin{proof}
This follows at once from the definition.
\end{proof}

\paragraph If $M$ is a left $A$-module, $f:M\to M$ a $\delta$-operator and
$\epsilon:P_\bullet\to M$ a projective resolution of~$M$, 
  \[
  \begin{tikzcd}
  \cdots \arrow[r]
    & P_2 \arrow[r, "d_2"]
    & P_1 \arrow[r, "d_1"]
    & P_0 \arrow[r, "\epsilon"]
    & M \arrow[r]
    & 0
  \end{tikzcd}
  \]
a \newterm{$\delta$-lifting} of~$f$ to $P_\bullet$  is a sequence
$f_\bullet=(f_i)_{i\geq0}$ of $\delta$-operators $f_i:P_i\to P_i$ such that
the diagram
  \[
  \begin{tikzcd}
  \cdots \arrow[r]
    & P_2 \arrow[r, "d_2"] \arrow[d, "f_2"]
    & P_1 \arrow[r, "d_1"] \arrow[d, "f_1"]
    & P_0 \arrow[r, "\epsilon"] \arrow[d, "f_0"]
    & M \arrow[r] \arrow[d, "f"]
    & 0
    \\
  \cdots \arrow[r]
    & P_2 \arrow[r, "d_2"]
    & P_1 \arrow[r, "d_1"]
    & P_0 \arrow[r, "\epsilon"]
    & M \arrow[r]
    & 0
  \end{tikzcd}
  \]
is commutative.

\paragraph As one can hope, $\delta$-liftings exist and are unique up to
reasonable equivalence. The key point to establishing this is the following
result:

\begin{Lemma*}\plabel{lemma:lift-1}
If $\epsilon:P\to M$ is a surjective morphism of left $A$-modules with
projective domain and $f:M\to M$ is a $\delta$-operator, then there exists
a $\delta$-operator $\tilde f:P\to P$ such that the diagram
  \[
  \begin{tikzcd}
  P \arrow[r, two heads, "\epsilon"] \arrow[d, "\tilde f"]
    & M \arrow[d, "f"]
    \\
  P \arrow[r, two heads, "\epsilon"]
    & M
  \end{tikzcd}
  \]
is commutative, $\tilde f(\ker\epsilon)\subseteq\ker\epsilon$ and the restriction
$f|_{\ker\epsilon}:\ker\epsilon\to\ker\epsilon$ is a $\delta$\nbd-operator.
\end{Lemma*}

\begin{proof}
Let $(p_i,\phi_i)_{i\in I}$ be a projective basis for~$P$, so that $p_i\in
P$ and $\phi_i\in\hom_A(P,A)$ for all $i\in I$, and for each $p\in P$ the
set $\{i\in I:\phi_i(p)\neq0\}$ is finite and $p=\sum_{i\in
I}\phi_i(p)p_i$, and let $(q_i)_{i\in I}$ be a family of elements of~$P$
such that $\epsilon(q_i)=f(\epsilon(p_i))$ for all $i\in I$. The function $\tilde
f:P\to P$ such
  \[
  \tilde f(p) = \sum_{i\in I}\bigl(\phi_i(p)q_i + \delta(\phi_i(p))p_i\bigr)
  \]
for all~$p\in P$ is easily seen to satisfy the conditions of the lemma.
\end{proof}

\paragraph We can now deduce in the usual way the existence and uniqueness
of $\delta$-liftings:

\begin{Lemma*} 
Let $M$ be a left $A$-module and let $\epsilon:P_\bullet\to M$ be a
projective resolution.
\begin{thmlist}

\item There exists a $\delta$-lifting $f_\bullet:P_\bullet\to P_\bullet$
of~$f$ to~$P_\bullet$.

\item If $f_\bullet$,~$f'_\bullet$ are $\delta$-liftings of
a~$\delta$-operator $f:M\to M$ to~$P_\bullet$, then $f_\bullet$ and
$f'_\bullet$ are $A$-linearly homotopic.

\end{thmlist}
\end{Lemma*}

\begin{proof}
The first part follows inductively using the result of
Lemma~\pref{lemma:lift-1} at each step. On the other hand, in the situation
of the second part the diagram
  \[
  \begin{tikzcd}
  \cdots \arrow[r]
    & P_2 \arrow[r, "d_2"] \arrow[d, "f_2-f_2'"]
    & P_1 \arrow[r, "d_1"] \arrow[d, "f_1-f_1'"]
    & P_0 \arrow[r, "\epsilon"] \arrow[d, "f_0-f_0'"]
    & M \arrow[r] \arrow[d, "0"]
    & 0
    \\
  \cdots \arrow[r]
    & P_2 \arrow[r, "d_2"]
    & P_1 \arrow[r, "d_1"]
    & P_0 \arrow[r, "\epsilon"]
    & M \arrow[r]
    & 0
  \end{tikzcd}
  \]
is commutative and the vertical arrows are morphisms of left $A$-modules,
so the morphism of complexes $f_\bullet-f'_\bullet:P_\bullet\to P_\bullet$
is homotopic to the zero morphism, through an $A$-linear homotopy.
\end{proof}

\paragraph Let now $M$ be a left $A$-module, $f:M\to M$ a
$\delta$-operator, $\epsilon:P_\bullet\to M$ a projective resolution and
$f_\bullet:P_\bullet\to P_\bullet$ a $\delta$-lifting of~$f$
to~$P_\bullet$. If $i\geq0$ and $\phi\in\hom_A(P_i,M)$, then the map
$f_i^\sharp(\phi):P_i\to M$ given by
  \[
  f_i^\sharp(\phi)(p) = f(\phi(p)) - \phi(f_i(p))
  \]
for all $p\in P_i$ is a morphism of $A$-modules, so we have a function 
  \[ 
  f_i^\sharp:\hom_A(P_i,M)\to\hom_A(P_i,M)
  \]
which is linear. A computation shows that, in fact, we obtain in this way a
morphism of complexes of vector spaces
  \[ \label{eq:fsharp}
  f_\bullet^\sharp:\hom_A(P_\bullet,M)\to\hom_A(P_\bullet,M). 
  \]

\paragraph To study the dependence of the morphism~$f_\bullet^\sharp$ on
the data used in its construction we will need the following observation.

\begin{Lemma*}\plabel{lemma:comm}
Let $M$ be a left $A$-module, $f:M\to M$ a~$\delta$-operator on~$M$,
$\epsilon:P_\bullet\to M$ and $\epsilon':P_\bullet'\to M$ projective
resolutions and $f_\bullet:P_\bullet\to P_\bullet$ and
$f'_\bullet:P'_\bullet\to P'_\bullet$ $\delta$-liftings of~$f$
to~$P_\bullet$ and to~$P'_\bullet$, respectively. If
$\alpha_\bullet:P'_\bullet\to P_\bullet$ is a morphism of complexes of
$A$-modules lifting $\id_M:M\to M$, then the diagram
  \[
  \begin{tikzcd}
  \hom_A(P_\bullet,M) \arrow[d, swap, "\alpha_\bullet^*"] \arrow[r, "f_\bullet^\sharp"]
    & \hom_A(P_\bullet,M) \arrow[d, "\alpha_\bullet^*"] 
    \\
  \hom_A(P'_\bullet,M) \arrow[r, "f_\bullet'^\sharp"]
    & \hom_A(P'_\bullet,M)
  \end{tikzcd}
  \]
commutes up to homotopy.
\end{Lemma*}

\begin{proof}
The difference $h_\bullet=\alpha_\bullet
f'_\bullet-f_\bullet\alpha_\bullet:P'_\bullet\to P_\bullet$ is, in
principle, only a morphism of complexes of vector spaces, but a computation
shows that its components are in fact $A$-linear. As
$h_\bullet:P'_\bullet\to P_\bullet$ is then a lifting of the zero map
$0:M\to M$, it is $A$-linearly homotopic to zero and, therefore, the
induced map
  \[
  h_\bullet^*:\hom_A(P_\bullet,M)\to\hom_A(P'_\bullet,M)
  \]
is also homotopic to zero. Since $\alpha_\bullet^*\circ
f_\bullet^\sharp-f_\bullet'^\sharp\circ\alpha_\bullet^*=h_\bullet^*$, the
lemma follows from this.
\end{proof}

\paragraph As a first consequence of this lemma, we see that if $M$ is a
left $A$-module, $\delta:M\to M$ a $\delta$-operator on~$M$,
$\epsilon:P_\bullet\to M$ a projective resolution and
$f_\bullet:P_\bullet\to P_\bullet$ and $f'_\bullet:P_\bullet\to P_\bullet$
$\delta$\nbd-liftings of~$f$ to~$P_\bullet$, then the maps of complexes
$f_\bullet^\sharp$,~$f_\bullet'^\sharp:\hom_A(P_\bullet,M)\to\hom_A(P_\bullet,M)$
are homotopic ---this is the special case of the lemma in which
$P'_\bullet=P_\bullet$, $\epsilon'=\epsilon$ and
$\alpha_\bullet=\id_{P_\bullet}$--- and therefore they induce the same map
on cohomology. We may therefore denote this induced map, which depends only
on~$f$ and not on the choice of the $\delta$-lifting used to compute it,
simply by
  \[
  \nabla^\bullet_{f,P_\bullet}:H(\hom_A(P_\bullet,M))\to H(\hom_A(P_\bullet,M)).
  \]
Next, if $\epsilon':P'_\bullet\to M$ is another projective resolution,
$f'_\bullet:P'_\bullet\to P'_\bullet$ a $\delta$-lifting of~$f$
to~$P'_\bullet$ and $\alpha_\bullet:P'_\bullet\to P_\bullet$ a lifting
of~$\id_M:M\to M$, the diagram 
  \[
  \begin{tikzcd}
  H(\hom_A(P_\bullet,M)) \arrow[d, swap, "H(\alpha_\bullet^*)"] 
        \arrow[r, "\nabla^\bullet_{f,P_\bullet}"]
    & H(\hom_A(P_\bullet,M)) \arrow[d, "H(\alpha_\bullet^*)"] 
    \\
  H(\hom_A(P'_\bullet,M)) 
        \arrow[r, "\nabla^\bullet_{f,P'_\bullet}"]
    & H(\hom_A(P'_\bullet,M))
  \end{tikzcd}
  \]
induced on cohomology by the one in the statement of the lemma
\emph{commutes}. Recalling the way the Yoneda functor $\Ext_A^\bullet(M,M)$
is identified with a derived functor, we see that the end result of all we
have done is the following:

\begin{Theorem}\label{thm:nabla}
If $M$ is a left $A$-module and $f:M\to M$ is a $\delta$-operator on~$M$,
there is a canonical morphism of graded vector spaces
  \[
  \nabla_f^\bullet:\Ext_A^\bullet(M,M)\to\Ext_A^\bullet(M,M)
  \]
such that for each projective resolution $\epsilon:P_\bullet\to M$ and each
$\delta$-lifting $f_\bullet:P_\bullet\to P_\bullet$ of~$f$ to~$P_\bullet$
the diagram 
  \[
  \begin{tikzcd}[column sep=1.5cm]
  H(\hom_A(P_\bullet,M)) \arrow[d, swap, "\cong"] 
        \arrow[r, "\nabla^\bullet_{f,P_\bullet}"]
    & H(\hom_A(P_\bullet,M)) \arrow[d, "\cong"] 
    \\
  \Ext_A^\bullet(M,M) 
        \arrow[r, "\nabla^\bullet_f"]
    & \Ext_A^\bullet(M,M) 
  \end{tikzcd}
  \]
in which the vertical arrows are the canonical isomorphisms, commutes.~\qed
\end{Theorem}

\paragraph In keeping with a long standing tradition, the very first
example we present is a somewhat trivial one, reserving for the next
section the one in which we are really interested.

\begin{Lemma*}
Suppose that $\delta:A\to A$ is an inner derivation, so that there exists
an $r\in A$ with $\delta(a)=[r,a]$ for all $a\in A$. If $M$ is a left
$A$-module and $f:M\to M$ is a $\delta$\nbd-operator on~$M$, then there
exists a morphism of left $A$-modules $\bar f:M\to M$ such that $f(m)=\bar
f(m)+rm$ for all $m\in M$, and the map
$\nabla^\bullet_f:\Ext_A^\bullet(M,M)\to\Ext_A^\bullet(M,M)$ is such that
for all $\phi\in\Ext_A^\bullet(M,M)$ we have $\nabla^i_f(\phi) = \bar
f_*(\phi)-\bar f^*(\phi) = [\bar f,\phi]$. 
\end{Lemma*}

\smallskip

Here $\bar f^*$, $\bar f_*:\Ext_A^\bullet(M,M)\to\Ext_A^\bullet(M,M)$ are
the maps induced by~$\bar f$ on the first and on the second argument
of~$\Ext$, respectively, and the commutator in the last formula is the one
obtained on $\Ext_A^\bullet(M,M)$ from the Yoneda product.

\smallskip

\begin{proof}
Let $\lambda:M\to M$ be the map given by multiplication by~$r$. A
computation shows that $\bar f=f-\lambda:M\to M$ is a morphism of left
$A$-modules. Let now $\epsilon:P_\bullet\to M$ be a projective resolution,
let $\bar f_\bullet:P_\bullet\to P_\bullet$ be a lifting of $\bar f$
to~$P_\bullet$ and let $\lambda_\bullet:P_\bullet\to P_\bullet$ be the map
given by multiplication by~$r$. Then $f_\bullet=\bar
f_\bullet+\lambda_\bullet:P_\bullet\to P_\bullet$ is a $\delta$-lifting
of~$f$ to~$P_\bullet$ and if $i\geq0$ and $\phi\in\hom_A(P_i,M)$, we have
$f_i^\sharp(\phi) = \bar f_*(\phi)-\bar f_i^*(\phi)$. This proves the
lemma.
\end{proof}

\section{The Gerstenhaber bracket on Hochschild cohomology}
\label{sect:hh}

\paragraph\label{p:bar} As we did in the previous section, we fix an
algebra~$A$ and a derivation $\delta:A\to A$. If~$M$~is a left $A$-module,
there is a standard projective resolution $\epsilon:B(M)_\bullet\to A$,
called the \newterm{bar resolution}, with $B(M)_i=A^{\otimes(i+1)}\otimes
M$ for each $i\geq0$, differentials given by
  \begin{multline*}
  d_i(a_0\otimes\cdots\otimes a_i\otimes m)
    = \sum_{i=0}^{i-1}(-1)^i
        a_0\otimes\cdots\otimes 
        a_{i-1}\otimes a_ia_{i+1}\otimes a_{i+2}\otimes\cdots\otimes a_i\otimes m \\[-5pt]
        + (-1)^ia_0\otimes\cdots\otimes a_{i-1}\otimes a_im
  \end{multline*}
for all $i\geq1$ and augmentation map $\epsilon:B(M)_0\to M$ such that
$\epsilon(a_0\otimes m)=a_0m$. If~$f:M\to M$ is a $\delta$-operator on~$M$,
there is a canonical $\delta$-lifting $f_\bullet:B(M)_\bullet\to
B(M)_\bullet$ of~$f$ to~$B(M)_\bullet$ given by
  \begin{multline*}
  f_i(a_0\otimes\cdots\otimes a_i\otimes m)
    = \sum_{i=0}^{i-1}a_0\otimes\cdots\otimes 
                a_{i-1}\otimes\delta(a_i)\otimes a_{i+1}\otimes\cdots\otimes a_i\otimes m  \\[-5pt]
        + a_0\otimes\cdots\otimes a_i\otimes f(m),
  \end{multline*}
as a computation will show. From this we obtain an explicit description of
the morphism $f_\bullet^\sharp:\hom_A(B(M)_\bullet,M)\to\hom_A(B(M)_\bullet,M)$: 
if $i\geq0$ and $\phi:B(M)_i\to M$ is $A$-linear, then
  \begin{multline*}
  f_i^\sharp(\phi)(a_0\otimes\cdots\otimes a_i\otimes m)
    = 
      f(\phi(a_0\otimes\cdots\otimes a_i\otimes m)) \\
      \shoveright{- \sum_{i=0}^{i-1}\phi(a_0\otimes\cdots\otimes 
       a_{i-1}\otimes\delta(a_i)\otimes a_{i+1}\otimes\cdots\otimes a_i\otimes m)}   \\
      - \phi(a_0\otimes\cdots\otimes a_i\otimes f(m))
  \end{multline*}

\paragraph\label{p:env} Let now $A^e=A\otimes A^\op$ be the enveloping
algebra of~$A$ ---so that we may identify $A$-bimodules with left
$A^e$-modules--- and consider $A$ as a left $A^e$-module as usual. From the
derivation $\delta:A\to A$ we can construct a new derivation
$\delta^e=\delta\otimes\id_A+\id_A\otimes\delta:A^e\to A^e$, and it turns
out that the map $\delta:A\to A$ is then a $\delta^e$-operator on~$A$.
Recalling that in our context we may identify $\Ext_{A^e}^\bullet(A,A)$
with the Hochschild cohomology $\HH^\bullet(A)$, our general construction
from the previous section gives us a map
  \[ \label{eq:nabla-delta}
  \nabla_\delta:\HH^\bullet(A)\to\HH^\bullet(A).
  \]
We want to see what this map is.

If for each $i\geq0$ we turn the left $A$-module $B(A)_i$ into an
$A$-bimodule with right action given by 
  \[
  a_0\otimes\cdots\otimes a_i\otimes a\cdot b = 
        a_0\otimes\cdots\otimes a_i\otimes ab,
  \]
then the projective resolution $\epsilon:B(A)_\bullet\to A$ of $A$ as a
left $A$-module constructed in~\pref{p:bar} becomes a projective resolution
of~$A$ as an $A^e$-module. Moreover, the $\delta$-lifting
$\delta_\bullet:B(A)_\bullet\to B(A)_\bullet$ of the $\delta$-operator
$\delta:A\to A$ constructed there is a $\delta^e$-lifting, as one can
easily check, so that we may use it to compute the
map~\eqref{eq:nabla-delta} up to the canonical identification of
$\HH^\bullet(A)$ with $H(\hom_{A^e}(B(A)_\bullet,A))$.

Let now $C^\bullet(A)$ be the standard complex which computes Hochschild
cohomology, which has $C^i(A)=\hom(A^{\otimes i},A)$ for each $i\geq0$ and
differentials $d^i:C^i(A)\to C^{i+1}(A)$ given by
  \begin{multline*}
  d^i(\phi)(a_1\otimes\cdots\otimes a_{i+1}) 
     = a_1\phi(a_2\otimes\cdots\otimes a_{i+1}) \\
     \shoveright{  + \sum_{j=1}^i(-1)^{j+1}\phi(
                a_1a\otimes\cdots\otimes a_{j-1}
                \otimes a_ja_{j+1}\otimes a_{j+2}\otimes\cdots\otimes
                a_{i+1})} \\
       + (-1)^{i+1}\phi(a_1\otimes\cdots\otimes a_i)a_{i+1}.
  \end{multline*}
Of course, there is an isomorphism of complexes
$\tau_\bullet:\hom_{A^e}(B(A)_\bullet,A)\to C^\bullet(A)$ such that
$\tau_i(\phi)(a_1\otimes\cdots\otimes a_i\otimes1)=\phi(1\otimes
a_1\otimes\cdots\otimes a_i)$ for all $i\geq0$ and all
$\phi\in\hom_{A^e}(B(A)_i,A)$. If now we let $[\place,\place]$ be the
Gerstenhaber bracket on~$C^\bullet(A)$, as constructed in~\cite{G}, then
the diagram
  \[
  \begin{tikzcd}
  \hom_{A^e}(B(A)_\bullet,A)
        \arrow[r, "\tau_\bullet"] \arrow[d, swap, "\delta_\bullet^\sharp"]
    & C^\bullet(A)
        \arrow[d, "{[\delta,\place]}"]
    \\
  \hom_{A^e}(B(A)_\bullet,A)
        \arrow[r, "\tau_\bullet"]
    & C^\bullet(A)
  \end{tikzcd}
  \]
commutes. This means that the map $\nabla_\delta$ of~\eqref{eq:nabla-delta}
is in fact simply $[\delta,\place]$. The point of all this is that
Theorem~\pref{thm:nabla} from the previous section tells us that we can
compute $\nabla_\delta$ using \emph{any} projective resolution of~$A$ as an
$A$-bimodule, provided we are able to construct a $\delta^e$-lifting of
$\delta$.

\paragraph If in the situation of the previous paragraph the
derivation~$\delta$ with which we start is inner, so that there is an~$r\in
A$ with $\delta=[r,\place]$, then $\delta^e$ is also inner, as
$\delta^e=[r^e,\place]$ with $r^e=r\otimes1-1\otimes r$. Moreover, if
$\epsilon:P_\bullet\to A$ is any resolution of~$A$ as an $A$-bimodule, then
there is a $\delta^e$-lifting $\delta_\bullet:P_\bullet\to P_\bullet$ of
$\delta:A\to A$ to~$P_\bullet$ such that for all $i\geq0$ and all $p\in
P_i$ we have $\delta_i(p)=rp-pr$. The associated map
$\delta_\bullet^\sharp:\hom_{A^e}(P_\bullet,A)\to\hom_{A^e}(P_\bullet,A)$
is identically zero, so we have that
$\nabla_\delta:\HH^\bullet(A)\to\HH^\bullet(A)$ itself is zero, as it
should.

\section{Examples}
\label{sect:examples}

\subsection*{Monomial algebras}

\paragraph Let $Q=(Q_0,Q_1,s,t)$ be a finite quiver and let $\kk Q$ be the
corresponding path algebra; if $v\in Q_0$ is a vertex, we write $e_v$ the
corresponding idempotent. Let $R$ be a set of paths in~$Q$ of length at
least~$2$ such that no element of~$R$ divides another and consider the
monomial algebra $A=\kk Q/(R)$. We write $E$ the subalgebra of~$A$ spanned
by the vertices; whenever $Z$ is a set of paths in~$Q$, the vector
space~$\kk Z$ which has~$Z$ as a basis has a natural structure
of~$E$-bimodule. 

We let $\epsilon:\Br_\bullet\to A$ be the projective resolution of $A$ as
an $A$-bimodule constructed by Michael Bardzell in~\cite{B}; a useful
companion to Bardzell's paper is the work \cite{Sk} of Emil Sk\"oldberg,
where a contracting homotopy on $\Br_\bullet$ is exhibited. There is a
sequence $(R_i)_{i\geq0}$ of sets of paths in~$Q$ such that $R_0=Q_0$ is
the set of vertices, $R_1=Q_1$ is the set of arrows, $R_2=R$ and
$\Br_i=A\otimes_E\kk R_i\otimes_EA$ for all $i\geq0$. The augmentation
$\epsilon:\Br_0=A\otimes_EA\to A$ is the map induced by the multiplication
of~$A$, and for each $i\geq1$ the differential $d:\Br_i\to\Br_{i-1}$ has
the property that whenever $u$ and $w$ are paths in~$Q$ which are not in
the ideal~$(R)$ and $v\in R_i$, and we can form the concatenation $uvw$, then
$d_i(u\otimes v\otimes w)$ is a $\kk$-linear combination of elementary
tensors $u'\otimes v'\otimes w'$ of $\Br_{i-1}=A\otimes_ER_{i-1}\otimes_EA$
with $u'$ and $w'$ paths in $Q$ not in~$(R)$ and a path~$v'$ in~$R_{i-1}$
such that the concatenation $u'v'w'$ exists and coincides with the path
$uvw$. 

The differentials $d:\Br_1\to\Br_0$ and $d:\Br_2\to\Br_1$, in particular,
are such that $d_1(1\otimes\alpha\otimes1)=\alpha\otimes1-1\otimes\alpha$
whenever $\alpha$ is an arrow, and 
  \[
  d_2(1\otimes r\otimes 1)
    =\sum_{i=1}^n\alpha_1\cdots\alpha_{i-1}\otimes\alpha_i\otimes \alpha_{i+1}\cdots\alpha_n
  \]
whenever $r=\alpha_1\cdots\alpha_n$ is an element of~$R$ of length~$n$. It
follows from this that we can identify $\hom_{A^e}(\Br_1,A)$ with the
vector space $\hom_{E^e}(\kk Q_1,A)$, whose elements are the linear
functions $\delta:\kk Q_1\to A$ which map each arrow~$\alpha$ to a linear
combination of paths in~$Q$ which are parallel to~$\alpha$. Such a map is
in the kernel of the differential
$d_2^*:\hom_{A^e}(\Br_1,A)\to\hom_{A^e}(\Br_2,A)$ iff for each element
$r=\alpha_1\cdots\alpha_n$ of~$R$ we have

  \[
  \sum_{i=1}^n\alpha_1\cdots\alpha_{i-1}\delta(\alpha_i)\alpha_{i+1}\cdots\alpha_n
    = 0
  \]
in the algebra~$A$, and indeed this condition is satisfied iff the
map~$\delta$ can be extended to a $E$-linear derivation $A\to A$.

\paragraph\label{p:diagonal} If $c:Q_1\to\kk$ is an arbitrary function
defined on the set of arrows of~$Q$, we may consider the extension
$c:Q_*\to\kk$ such that $c(w)=\sum_{i=1}^nc(\alpha_i)$ for each path
$w=\alpha_1\cdots\alpha_n$ in~$Q$, and then the $E^e$-linear map
$\delta_c:A\to A$ such that $\delta_c(w)=c(w)w$ for all paths in~$Q$. One
sees immediately that this is an $E^e$-linear derivation. We say that
derivations of the form~$\delta_c$ for some $c:Q_1\to\kk$ are
\newterm{diagonal}.

We consider the derivation
$\delta_c^e=\delta_c\otimes1+1\otimes\delta_c:A^e\to A^e$ on the enveloping
algebra~$A^e$, as in~\pref{p:env}, and view $\delta_c:A\to A$ as a
$\delta_c^e$-operator on the left $A^e$-module~$A$. There is a
$\delta_c^e$-lifting $(\delta_c)_\bullet:\Br_\bullet\to\Br_\bullet$
of~$\delta_ c$ to the resolution~$\Br_\bullet$ such that for each $i\geq0$,
each $u$,~$w\in Q_*$ and each $v\in R_i$ such that the concatenation $uvw$
exists, we have
  \[
  (\delta_c)_i(u\otimes v\otimes w) = c(uvw)\cdot u\otimes v\otimes w.
  \]
We may use this $\delta_c^e$-lifting to compute the Gerstenhaber bracket
$[\delta_c,\place]$ on~$\HH^\bullet(A)$, if we identify it with the
cohomology of the complex $\hom_{A^e}(\Br_\bullet,A)$. Indeed, the space
$\hom_{A^e}(\Br_i,A)$ of $i$-cochains in this complex can be identified
with $\hom_{E^e}(\kk R_i,A)$, and clearly has as a basis the set of all
$E^e$-linear maps $\phi_{r,u}:\kk R_i\to A$ with $r\in R_i$ and $u$ a
path in~$Q$ parallel to~$r$ and which is non-zero in~$A$, given by
$\phi_{r,u}(s)=0$ for all $s\in R_i\setminus\{r\}$ and $\phi_{r,u}(r)=u$,
so it is sufficient to compute $[\delta_c,\phi_{r,u}]$, and this is,
according to what we have done so far,
  \[ \label{eq:diag}
  [\delta_c,\phi_{r,u}]
    = \nabla_{\delta_c}(\phi_{r,u})
    = (\delta_c)_i^\sharp(\phi_{r,u})
    = \bigl(c(u)-c(r)\bigr)\phi_{r,u}.
  \]
If the set~$R$ satisfies the condition that
  \[
  \claim{whenever $\alpha:i\to j$ is an arrow  of~$Q$ we have $\dim
  e_iAe_j=1$, so that there are no non-zero paths in~$A$ parallel to an
  arrow apart from the arrow itself,} 
  \]
then it is easy to see that \emph{all} elements of $\HH^1(A)$ are
represented by diagonal derivations and in this situation Lucrecia Rom\'an
has obtained the formula~\eqref{eq:diag} in her thesis~\cite{R} after
fearlessly computing comparison morphisms $\Br_\bullet\rightleftarrows
B(A)_\bullet$ and then using the usual formula for the Gerstenhaber bracket
on the standard complex~$C^\bullet(A)$.

\paragraph Let us suppose now that we have an integer $N\geq2$ and that $R$
is the set~$Q_N$ of \emph{all} paths of length~$N$ in~$Q$; the algebra~$A$
is then what is usually called a \newterm{truncated algebra}. In this case
the Bardzell resolution~$\Br_\bullet$ admits a very simple description,
which we now recall.
Let $\jump:\NN_0\to\NN_0$ be the function such that $\jump(2k)=Nk$ and
$\jump(2k+1)=Nk+1$ for all $k\in\NN_0$. Then for all $i\geq0$ we have 
$R_i=Q_{\jump(i)}$, the set of all paths of length~$\jump(i)$ in the quiver~$Q$, and
the differential on~$\Br_i$ is such that for each $w\in Q_{\zeta(i)}$ we
have
  \begin{align}
  & d(1\otimes w\otimes 1) = 
    \sum_{\substack{aub=w\\u\in Q_{N(k-1)+1}}} a\otimes u\otimes b 
    && \text{if $i=2k$ is even}
\shortintertext{and}
  & d(1\otimes w\otimes 1) = 
    a\otimes r\otimes 1 - 1\otimes l\otimes b
    && \text{if $i=2k+1$ is odd,} 
  \end{align}
where in this last case $a$,~$b\in Q_1$ and $r$,~$l\in Q_{Nk}$ are such that $w=ar=lb$.

\begin{Lemma*}
An $E^e$-linear map $\delta:\kk Q_1\to A$ is a $1$-cocycle in the complex
$\hom_{A^e}(\Br_\bullet,A)$ if it takes values in~$\rad A$, and if the
quiver $Q$ has more than one vertex and is connected then this condition is
also necessary.
\end{Lemma*}

\begin{proof}
The sufficiency is clear, so we deal only with the second part.
Let $\delta:\kk Q_1\to A$ be a $1$-cocycle and let
$d_2^*:\hom_{A^e}(\Br_1,A)\to\hom_{A^e}(\Br_2,A)$ be the differential. There are
$E^e$-linear maps $\delta_0:\kk Q_1\to\kk E$ and $\delta_+:\kk Q_1\to\rad
A$ such that $\delta=\delta_0+\delta_+$. If $r=\alpha_1\cdots\alpha_N\in R$, 
then
$d_2^*(\delta_+)(r)=\sum_{i=1}^N\alpha_1\cdots\delta_+(\alpha_i)\cdots\alpha_N$
is a linear combination of paths of length at least~$N$, so that in fact
$d_2^*(\delta_+)=0$ and, in particular, we have $d_2^*(\delta_0)=0$. 
On the other hand, if we write $\Omega Q$ the set of loops in~$Q$,
there is a function $\lambda:\Omega Q\to\kk$ such that for all $\alpha\in
Q_1$ we have
  \[
  \delta_0(\alpha) =
    \begin{cases*}
    \lambda(\alpha)s(\alpha), & if $\alpha\in\Omega Q$; \\
    0, & otherwise.
    \end{cases*}
  \]
If $\alpha\in\Omega Q$, then $\alpha^N\in R$ and
$d_2^*(\delta_0)(\alpha^N)=N\lambda(\alpha)\alpha^{N-1}=0$, so that
$N\lambda(\alpha)=0$. As $Q$ has more than one vertex and is
connected, there is an arrow $\beta\in Q_1\setminus\Omega Q$ such that one
of $\alpha^{N-1}\beta$ or~$\beta\alpha^{N-1}$ is in~$R$, and then either
$d_2^*(\delta_0)(\alpha^{N-1}\beta)=(N-1)\lambda(\alpha)\alpha^{N-2}\beta=0$
or
$d_2^*(\delta_0)(\beta\alpha^{N-1})=(N-1)\lambda(\alpha)\beta\alpha^{N-2}=0$,
and therefore $(N-1)\lambda(\alpha)=0$. It follows that
$\lambda(\alpha)=0$ and we see that in fact $\delta_0=0$.
\end{proof}

\smallskip

We now fix an $E^e$-linear map $\delta:\kk Q_1\to\rad A$ and assume
moreover that $\delta$ is \newterm{homogeneous},
so that there is an $l\in\{1,\dots,N-1\}$ such that the image
of~$\delta$ is in $\kk Q_l$; the \newterm{degree} of~$\delta$ is then the
number $l-1$. We will write the $E^e$-linear derivation
$A\to A$ which extends the $1$-cocycle~$\delta$ by the same letter. If
$n$,~$m\geq0$, there is a unique $E^e$-linear map $\Delta_n^m:\kk Q\to
A\otimes_E\kk Q_m\otimes_E A$ such that for each path $w\in Q_*$ we have
$\Delta_n^m(w)=0$ if $\abs{w}<n+m$ and
  \[
  \Delta_n^m(w) = a\otimes u\otimes b,
  \]
with $aub$ the unique factorization of~$w$ with $\abs{a}=n$ and $\abs{u}=m$.
A verification shows that there is a $\delta^e$-lifting $\delta_\bullet:\Br_\bullet\to\Br_\bullet$ 
of~$\delta$ to the complex $\Br_\bullet$ such that for each $i\geq0$ and
each $w=\alpha_1\cdots\alpha_{\zeta(i)}\in Q_{\zeta(i)}$ we have
  \[
  \delta_i(1\otimes w\otimes 1) = 
    \begin{dcases*}
    \Delta_{l-1}^{\zeta(i)}(\delta(w)), 
                & if $i$ is even; \\
    \Delta_{l-1}^{\zeta(i)}(\delta(w))
        + \sum_{j=0}^{l-2}
                \Delta_j^{\zeta(i)}(\alpha_1\cdots\alpha_{\zeta(i)-1}\delta(\alpha_{\zeta(i)}), 
                & if $i$ is odd.
    \end{dcases*}
  \]
To exemplify how we can use this, we propose to describe the Lie action
of~$\HH^1(A)$ on the cohomology~$\HH^\bullet(A)$. To do this we need some
information on this cohomology, of course. In low degrees, it was computed
by Claude Cibils in~\cite{Cibils:truncated} and this was extended to all
degrees by Ana Locateli in~\cite{Locateli} when the ground field has
characteristic zero and by Yunge Xu, Yang Han and Wenfeng Jiang for the
general case in~\cite{XHJ}. What these authors do, though, is to find the
dimensions of the cohomology groups and for our purpose this is not
enough, as we need the actual cocycles. To keep things simple, we will
content ourselves with a special case and with obtaining information only on the even part
of cohomology.

\begin{Lemma*}
Suppose that $Q$ has no sources and no sinks and that it is not an oriented
cycle. If $k\geq1$, then the subspace of~$2k$-cocycles in the
complex~$\hom_{A^e}(\Br_\bullet,A)$ is $\hom_{E^e}(\kk Q_{Nk},\kk
Q_{N-1})$.
\end{Lemma*}

\smallskip

Here the simplifying assumption is that there are no sources or sinks,
while the exclusion of the case of an oriented cycle is due to the fact
that this is really an exceptional case. Using this lemma we can easily
obtain the promised result:

\begin{Corollary*}
A homogeneous derivation of nonzero degree acts by zero on 
the even part $\HH^{\mathsf{even}}(A)$ of the Hochschild cohomology.
\end{Corollary*}

\begin{proof}
It is enough to check that if $l\in\{2,\dots,N-\}$ and $\delta:\kk Q_1\to\kk Q_l$
is a homogeneous cocycle of degree $l-1$ and $\delta_\bullet$ is the
lifting of~$\delta$ to the complex~$\Br_\bullet$ described above, then
$\delta_{2k}^\sharp(\phi)=0$ for all $k\geq1$ and all 
$\phi\in\hom_{E^e}(\kk Q_{Nk},\kk
Q_{N-1})$, since the lemma tells us that these are all the $2k$-cocycles.
This is a trivial computation.
\end{proof}

\smallskip

\begin{proof}[Proof of the lemma]
If $\phi:\kk Q_{Nk}\to A$ is a $2k$-cochain in that complex, there are $E^e$-linear
maps
$\phi_i:\kk Q_{Nk}\to\kk Q_l$ for each $i\in\{0,\dots,N-1\}$ such that
$\phi=\sum_{i=0}^{N-1}\phi_i$, and it is clear, since the algebra is
monomial, that $\phi$ is a cocycle if and only if all the $\phi_i$ are.
Now, an $E^e$-linear map $\kk Q_{Nk}\to\kk Q_{N-1}$ is automatically a
cocycle, so the lemma will be proved if we show that 
  \[
  \claim{if $0\leq l<N-1$ and $\phi:\kk Q_{Nk}\to\kk Q_l$ is an 
  $E^e$-linear map which is a $2k$-cocycle in the
  complex~$\hom_{A^e}(\Br_\bullet,A)$, then $\phi=0$.}
  \]
To do that, let us fix an integer~$l$ such that $0\leq l<N-1$ and an
$E^e$-linear map $\phi:\kk Q_{Nk}\to\kk Q_l$ which is a $2k$-cocycle with
values in~$\kk Q_l$. This means, precisely, that for each path
$\alpha_1\dots\alpha_{Nk+1}\in Q_{Nk+1}$ we have
  \[ \label{eq:t:1}
  \phi(\alpha_1\cdots\alpha_{Nk})\alpha_{Nk+1}
    = \alpha_1\phi(\alpha_2\cdots\alpha_{Nk+1}).
  \]
Let $r$ be an integer such that $0\leq 2r\leq l$ and suppose that 
  \[ \label{eq:t:2}
  \claim{for each $w=\alpha_1\cdots\alpha_{Nk}\in Q_{Nk}$ there exists a
  $\bar\phi(w)\in\kk Q_{l-2r}$ which is a linear combination of paths from
  $s(\alpha_{r+1})$ to $t(\alpha_{Nk-r})$
  such that
  $\phi(w)=\alpha_1\cdots\alpha_r\bar\phi(w)\alpha_{Nk-r+1}\cdots\alpha_{Nk}$.}
  \]
Notice that this holds when $r=0$. Let $w=\alpha_1\cdots\alpha_{Nk}\in
Q_{Nk}$. As there are no sinks in~$Q$, there is an arrow $\alpha_{Nk+1}\in
Q_1$ such that $w\alpha_{Nk+1}$ is a path. If we put
$w'=\alpha_2\cdots\alpha_{Nk+1}$, then we have from~\eqref{eq:t:1} 
and the hypothesis~\eqref{eq:t:2} that
  \[
  \alpha_1\cdots\alpha_r\bar\phi(w)\alpha_{Nk-r+1}\cdots\alpha_{Nk+1}
    = \alpha_1\cdots\alpha_{r+1}\bar\phi(w')\alpha_{Nk-r+2}\cdots\alpha_{Nk+1},
  \]
so that in fact
  \[ \label{eq:t:3}
  \bar\phi(w)\alpha_{Nk-r+1} = \alpha_{r+1}\bar\phi(w').
  \]
If $l-2r\geq2$, this implies that all the paths appearing in $\bar\phi(w)$
start with $\alpha_{r+1}$ and, by symmetry, they also end in
$\alpha_{Nk-r}$. In other words, there exists a $\dbar\phi(w)\in\kk
Q_{l-2r-2}$ which is a sum of paths from $s(\alpha_{r+2})$ to
$t(\alpha_{Nk-r-1})$ such that
$\bar\phi(w)=\alpha_{r+1}\dbar\phi(w)\alpha_{Nk-r}$. In this case we have
that the condition~\eqref{eq:t:2} holds with $r$ replaced with~$r+1$, and
we may therefore proceed inductively.

If instead $l-2r=1$, then the equation~\eqref{eq:t:3} tells us there exists
a scalar $\lambda(w)\in\kk$ which is zero if $\alpha_{r+1}\neq \alpha_{Nk-r}$
such that $\bar\phi(w)=\lambda(w)\alpha_{r+1}$,
so in this case we have
  \(
  \phi(w)=\lambda(w)\,\alpha_1\cdots\alpha_r\alpha_{r+1}\alpha_{Nk-r+1}\cdots\alpha_{Nk}
  \).
Finally, if $l-2r=0$, then the equation~\eqref{eq:t:3} implies that there
is a scalar~$\lambda(w)\in\kk$ which is zero if
$\alpha_{r+1}\neq\alpha_{Nk-r+1}$ and such that
$\bar\phi(w)=\lambda(w)e_{s(\alpha_{r+1})}$, and therefore
  \(
  \phi(w)=\lambda(w)\,\alpha_1\cdots\alpha_r\alpha_{Nk-r+1}\cdots\alpha_{Nk}
  \).

In this way we conclude that there is a function $\lambda:Q_{Nk}\to\kk$
such that for each $w=\alpha_1\cdots\alpha_{Nk}\in Q_{Nk}$ we have
  \[
  \phi(w) =
    \begin{cases*}
    \lambda(w)\,\alpha_1\cdots\alpha_r\alpha_{Nk-r+1}\cdots\alpha_{Nk},
        & if $l=2r$ is even; \\
    \lambda(w)\,\alpha_1\cdots\alpha_r\alpha_{r+1}\alpha_{Nk-r+1}\cdots\alpha_{Nk},
        & if $l=2r+1$ is odd;
    \end{cases*}
  \]
with 
  \[ \label{eq:t:5}
  \claim{$\lambda(w)=0$ if $\alpha_{r+1}\neq\alpha_{Nk-r+1}$ and $l=2r$ is
  even, or if $\alpha_{r+1}\neq\alpha_{Nk-r}$ and $l=2r+1$ is odd.}
  \]

We define a relation $\sim$ on the set~$Q_{Nk}$ so that for each
$w$,~$w'\in Q_{Nk}$ we have $w\sim w'$ iff there exist arrows
$\alpha$,~$\beta\in Q_1$ such that $w\alpha=\beta w'$, and let $\approx$ be
the least equivalence relation on~$Q_{Nk}$ coarser that~$\sim$. Now, if
$w$,~$w'\in Q_{Nk}$ are such that $w\sim w'$, there is a path
$\alpha_1\cdots\alpha_{Nk+1}\in Q_{Nk+1}$ such that
$w=\alpha_1\cdots\alpha_{Nk}$ and $w'=\alpha_2\cdots\alpha_{Nk+1}$. From
equation~\eqref{eq:t:1} we have 
  \[
  \lambda(w)\,
  \alpha_1\cdots\alpha_r\alpha_{Nk-r+1}\cdots\alpha_{Nk+1}
  =
  \lambda(w')\,
  \alpha_1\cdots\alpha_{r+1}\alpha_{Nk-r+2}\cdots\alpha_{Nk+1}
  \]
if $l=2r$ is even, and
  \[
  \lambda(w)\,
  \alpha_1\cdots\alpha_r\alpha_{r+1}\alpha_{Nk-r+1}\cdots\alpha_{Nk+1}
  =
  \lambda(w')\,
  \alpha_1\cdots\alpha_{r+1}\alpha_{r+2}\alpha_{Nk-r+2}\cdots\alpha_{Nk+1}
  \]
if $l=2r+1$ is odd. In any of the two cases we find that
$\lambda(w)=\lambda(w')$, and it follows from this that $\lambda$ is constant on the
equivalence classes of the relation~$\approx$.

We claim that in fact 
  \[ \label{eq:t:4}
  \claim{there is only one equivalence class for the relation~$\approx$.}
  \]
There is a preorder~$\preceq$ on the set of vertices~$Q_0$ such that
whenever $i$,~$j\in Q_0$ we have $i\preceq j$ iff there is a path in~$Q$
from~$j$ to~$i$, and associated to~$\preceq$ there is an equivalence
relation~$\asymp$ on~$Q_0$ such that $i\asymp j$ iff $i\preceq j$ and
$j\preceq i$. The equivalence classes of~$\asymp$ are the \newterm{strongly
connected components} of the quiver and the preorder~$\preceq$ induces an
actual order on the quotient $Q_0/\mathord\asymp$; in particular, we may
speak of maximal and minimal strongly connected components. Our
claim~\eqref{eq:t:4} now follows easily from the following two facts:
\begin{itemize}

\item If $w\in Q_{Nk}$, there is a path $w'\in Q_{Nk}$ which is totally
contained in a maximal strongly connected component of~$Q$ and such that
$w\approx w'$, and the same is true replacing `maximal' by `minimal'.

\item If $w$ and $w'$ are elements of~$Q_{Nk}$ totally contained in
possibly different maximal strongly connected components of~$Q$, then
$w\approx w'$.

\end{itemize}
Let us prove the first one, leaving the other for the reader. Let $w\in
Q_{Nk}$. Let $i\in Q_0$ be a vertex in a strongly connected component~$C$
of~$Q_0$ which is maximal among those elements in~$Q_0/\mathord\asymp$
greater than the one containing $s(w)$. As $i$ is not a source, there
exists an arrow $\alpha$ with $t(\alpha)=i$; since the component~$C$ is
maximal, there is a path $u$ from $i$ to $s(\alpha)$ which never leaves~$C$
and, since $\alpha u$ is a closed path starting at $i$, we see that there
exists a path $w'\in Q_{Nk}$ contained in~$C$ and ending in~$i$.
On the other hand, the choice
of~$i$ implies that there exists a path $w_1$ in~$Q$ going from $i$
to~$s(w)$. Considering all the factors of length~$Nk$ of the path $w'w_1w$,
we see at once that $w'\approx w$, as we wanted.

It follows now from~\eqref{eq:t:4} that the function $\lambda:Q_{Nk}\to\kk$
is constant. As $Q$ is not an oriented cycle, there exists a vertex $i\in
Q_0$ and two different arrows $\alpha$ and~$\alpha'$ such that either
$s(\alpha)=s(\alpha')=i$ or $t(\alpha)=t(\alpha')=i$. Suppose, for example,
that we are in the first case; the other can be handled in the same way. 
Since there are no sources and sinks in~$Q$, if $l=2r$ is even, there are
paths $u\in Q_{Nk-r}$, and $v$,~$v'\in Q_{r-1}$, and if $l=2r+1$ is odd,
there are paths $u\in Q_{Nk-r-1}$ and $v$,~$v'\in Q_{r}$ such that, in
either case, $w=u\alpha v$ and $w'=u\alpha'v'$ are paths of length~$Nk$.
In view of our observation~\eqref{eq:t:5}, at least one of the scalars~$\lambda(w)$
and~$\lambda(w')$ is zero. With this we can therefore conclude that
$\phi=0$.
\end{proof}

\paragraph In the presence of sinks and sources, the homogeneous
derivations of positive degree of a truncated algebra may well act
non-trivially, as the following simple example shows. Fix $N\geq3$ and
$k\geq4$, consider the quiver~$Q$
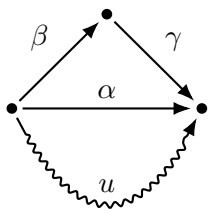
  \[
  \begin{tikzpicture}[
        thick, 
        shorten <=4pt, 
        shorten >=4pt, 
        -{Latex[]}, 
        decoration={snake,amplitude=1pt,segment length=4pt, pre length=1em, post length=1em}
        ]
    \coordinate (1) at (0,0);
    \coordinate (2) at (2.5,0);
    \coordinate (3) at (1.25,1.25);
    \fill (1) circle(2pt);
    \fill (2) circle(2pt);
    \fill (3) circle(2pt);
    \draw (1) -- node[above] {$\alpha$} (2);
    \draw (1) -- node[above left] {$\beta$} (3);
    \draw (3) -- node[above right] {$\gamma$} (2);
    \draw[decorate] (1) parabola bend (1.25,-1.25) (2);
    \node[above] at (1.25,-1.25) {$u$};
  \end{tikzpicture}
  \]
in which $\alpha$,~$\beta$ and~$\gamma$ are arrows and $u$ is a path of
length $\zeta(k)$, and let $A$ be the quotient of the path algebra~$\kk Q$ by the
ideal generated by~$Q_N$. Identifying~$\HH^\bullet(A)$ with the cohomology
of the complex~$\hom_{A^e}(\Br_\bullet,A)$, we see at once that
$\HH^0(A)\cong\kk$, that $\HH^1(A)$ is the vector space spanned by two
linearly independent diagonal derivations and the $E^e$-linear derivation
$\delta:A\to A$, homogeneous of degree~$1$, such that
$\delta(\alpha)=\beta\gamma$ and
$\delta(\omega)=0$ for all arrows~$\omega$ different
from~$\alpha$, that $\HH^{k}(A)$ is $2$-dimensional, spanned by the
$E^e$-linear maps $\phi_1$,~$\phi_2:\kk Q_{\zeta(k)}\to A$ such that
$\phi_1(u)=\alpha$ and $\phi_2(u)=\beta\gamma$, and that all other
cohomology groups are zero. Moreover, computing the Lie action
of~$\HH^1(A)$ of~$\HH^{k}(A)$ using the liftings constructed above for
truncated algebras shows immediately that $[\delta_2,\phi_0] = -\phi_1$. In
particular, the derivation~$\delta$ acts non-trivially
on~$\HH^{k}(A)$.

\subsection*{Crossed products}

\paragraph Let $A$ be an algebra and let $G$ be a finite group acting
on~$A$; we will suppose throughout that our ground field ring is a field
in which the order of~$G$ is invertible and which splits~$G$.
We may construct the cross product algebra $A\rtimes G$ which
as a vector space is~$A\otimes\kk G$, with $\kk G$ the group algebra
of~$G$, and where multiplication is such that $a\otimes g\cdot b\otimes
g=ag(b)\otimes gh$. 

The group acts diagonally on the enveloping algebra~$A^e$, so we can
consider also the crossed product $A^e\rtimes G$. An $A^e\rtimes G$-module
structure on a vector space~$M$ may be described as an $A^e$-module structure 
on which $G$ acts in a compatible way, in the sense that
$g(amb)=g(a)g(m)g(b)$ for all $g\in G$, $a$,~$b\in A$ and $m\in M$. If $M$
is such an $A^e\rtimes G$-module, we denote $M\rtimes G$ the $(A\rtimes
G)^e$-module with underlying vector space $M\otimes\kk G$ and left and
right actions by $A\rtimes G$ given by
  \begin{align}
  &  a\otimes g\cdot m\otimes h = ag(m)\otimes gh,
  && m\otimes h\cdot a\otimes g = mh(a)\otimes hg,
  \end{align}
respectively, whenever $a\otimes g\in A\rtimes G$ and $m\otimes h\in
M\rtimes G$. On the other hand, if $M$ is an $(A\rtimes G)^e$-module, we
denote $M^\ad$ the $A^e\rtimes G$-module which coincides with $M$ as an $A^e$-module 
and on which $G$ acts so that
  \[
  g\cdot m = gmg^{-1}
  \]
for all $g\in G$ and all $m\in M$. These two constructions are related in
the following way:

\begin{Lemma*}
If $M$ is an $A^e\rtimes G$-module and $N$ is an $(A\rtimes
G)^e$-module, there are isomorphisms
  \[
  \begin{tikzcd}
  \hom_{(A\rtimes G)^e}(M\rtimes G,N) \arrow[r, shift left=0.5ex, "\Phi"]
        & \hom_{A^e}(M,N^\ad)^G \arrow[l, shift left=0.5ex, "\Psi"]
  \end{tikzcd}
  \]
natural in~$M$ and~$N$.
\end{Lemma*}

\noindent On the right we are taking invariants with respect to the action
of~$G$ on~$\hom_{A^e}(M,N^\ad)$ given by $(g\cdot f)(m)=g\cdot
f(g^{-1}\cdot m)$ for each $g\in G$ and each $A^e$-linear map $f:M\to
N^\ad$.

\smallskip

\begin{proof}
We may put $\Phi(f)(m)=f(m\otimes 1)$ for all $f\in\hom_{(A\rtimes
G)^e}(M\rtimes G,N)$ and all $m\in M$, and $\Psi(f)(m\otimes g)=f(m)g$ for
all $f\in\hom_{A^e}(M,N^\ad)$ and all $m\otimes g\in M\rtimes G$.
\end{proof}

\smallskip

The following result is an immediate consequence of the lemma, since the
functor~$(\place)^G$ which computes invariants is exact.

\begin{Corollary*}
If $P$ is an
$A^e\rtimes G$-module which is projective as an $A^e$-module, then
$P\rtimes G$ is a projective $(A\rtimes G)^e$-module.~\qed
\end{Corollary*}

\smallskip

We view $A$ as an $A^e\rtimes G$-module in the obvious way and let
$P_\bullet\to A$ be a resolution of~$A$ as an $A^e\rtimes G$-module by
modules which are projective as $A^e$-modules; such a resolution exists:
for example, as $A^e\rtimes G$ is projective as a left $A^e$-module, it
suffices to take $P_\bullet$ to be an $A^e\rtimes G$-projective resolution
of~$A$, but one can often be much more economical. The corollary implies then
that $P_\bullet\rtimes G\to A\rtimes G$ is a projective resolution of
$A\rtimes G$ as an $(A\rtimes G)^e$-module. In particular the cohomology of
the complex $\hom_{(A\rtimes G)^e}(P_\bullet\rtimes G,A\rtimes G)$ can be
identified to the Hochschild cohomology $\HH^\bullet(A\rtimes G)$ of
$A\rtimes G$. As this complex is, according to the lemma,  naturally isomorphic to
$\hom_{A^e}(P_\bullet,A\rtimes G)^G$ and since $G$ acts semisimply, taking
homology in this second complex we see 
that $\HH^\bullet(A\rtimes G)$ is isomorphic to $H^\bullet(A,A\rtimes
G)^G$; this result is usually obtained using the spectral sequence
constructed by Drago\c{s} \c{S}tefan in~\cite{Stefan}, but for our purposes
we need the isomorphism to come out of an actual resolution.

Let $\delta:A\rtimes G\to A\rtimes G$ be a derivation of~$A\rtimes G$. The
restriction $\delta|_{\kk G}:\kk G\to A\rtimes G$  is a derivation of the
group algebra~$\kk G$ with values in the $\kk G$-bimodule $A\rtimes G$ and
therefore, since $\kk G$ is a separable algebra, this restriction is inner:
there exists an element $u\in A\rtimes G$ such that $\delta(g)=[u,g]$ for
all $g\in G$. It follows that the derivation $\delta-[u,\place]:A\rtimes
G\to A\rtimes G$, which is cohomologous to~$\delta$, is
\newterm{normalized}, that is, it vanishes on~$G$ and we conclude that,
up to inner derivations, we can assume that derivations
of~$A\rtimes G$ are normalized. 

\paragraph We specialize now the discussion to the following situation.
Let $G$ be a finite group, let $\rho:G\to\GL(V)$ be a representation
of~$G$ on a finite dimensional vector space~$V$, and consider the corresponding action of~$G$ on the symmetric
algebra $S(V)$ and the associated crossed product $S(V)\rtimes G$. We then
have available the bimodule Koszul resolution
$K_\bullet=S(V)\otimes\Lambda^\bullet V\otimes S(V)\to S(V)$ of~$S(V)$ as
an $S(V)^e$-module, and it turns out, when we endow $K_\bullet$ with
its natural action of~$G$, that the resolution is a complex of
$S(V)^e\rtimes G$-modules.

As explained above, the complex
$\hom_{(S(V)\rtimes G)^e}(K_\bullet\rtimes G,S(V)\rtimes G)$, which
computes the Hochschild cohomology~$\HH^\bullet(S(V)\rtimes G)$,
is isomorphic to
$\hom_{S(V)^e}(K_\bullet,S(V)\rtimes G)^G$
which, in turn, can be identified with the complex
$\hom(\Lambda^\bullet(V),S(V)\rtimes G)^G$. There is, moreover, a
decomposition
  \[ \label{eq:x:1}
  \hom(\Lambda^\bullet(V),S(V)\rtimes G) 
        = \bigoplus_{g\in G}\hom(\Lambda^\bullet(V),S(V)\rtimes g)
  \]
If $g\in G$, we let $V^g$ be the fixed subspace of~$g$ in~$V$, $V_g$ the
subspace of~$V$ spanned by eigenvectors of~$g$ corresponding to eigenvalues
different from~$1$, and $d(g)=\dim V_g$. As $V=V^g\oplus V_g$ and this
decomposition is preserved by~$g$, we have $S(V)=S(V^g)\otimes S(V_g)$, 
$\Lambda^\bullet(V)=\Lambda^\bullet(V^g)\otimes\Lambda^\bullet(V_g)$, and
there is an obvious map of complexes, 
  \[
  \vee:
  \hom(\Lambda^\bullet(V^g),S(V^g))
    \otimes\hom(\Lambda^\bullet(V_g),S(V_g)\rtimes g)
    \to \hom(\Lambda^\bullet(V),S(V)\rtimes g),
  \]
which can be seen to be a quasi-isomorphism; this is part of the content of
Theorem~\textsc{XI}.3.1 in \cite{CE}, for example. The complex
$\hom(\Lambda^\bullet(V^g),S(V^g))$ has zero differential; on the other
hand, the cohomology of the complex
$\hom(\Lambda^\bullet(V_g),S(V_g)\rtimes g)$ is zero except in
degree~$d(g)$, where it is one dimensional and spanned by the
cohomology class of any non-zero linear map
$\omega_g:\Lambda^{d(g)}(V_g)\to \kk g$, as shown in Lemma~3.4 of~\cite{F}; we fix one such map once and for
all. All this means that the inclusion
of complexes
  \[
  \hom(\Lambda^\bullet(V^g),S(V^g))\vee\omega_g
        \hookrightarrow\hom(\Lambda^\bullet(V),S(V)\rtimes g)
  \]
is a quasi-isomorphism, with the subcomplex having zero differential. Since
everything in sight is $G$-equivariant, and taking into account the
decomposition~\eqref{eq:x:1}, the same is true of the inclusion
  \[ \label{eq:t:x}
  \left(\bigoplus_{g\in G}\hom(\Lambda^\bullet(V^g),S(V^g))\vee\omega_g\right)^G
        \hookrightarrow
  \hom(\Lambda^\bullet(V),S(V)\rtimes G)^G.
  \]
If $g$,~$h$ are in~$G$, then $h\cdot\omega_g$ and $\omega_{hgh^{-1}}$ are two
non-zero elements of the $1$-dimensional vector space
$\hom(\Lambda^{d(hgh^{-1})}(V_{hgh^{-1}}),\kk hgh^{-1})$, so  there exists a scalar
$\lambda(h, g)\in\kk^\times$ such that
$h\cdot\omega=\lambda(h,g)\omega_{hgh^{-1}}$. In this way we find a function
$\lambda:G\times G\to\kk^\times$ and the associativity of the action of~$G$
implies that $\lambda(gh,k)=\lambda(h,k)\lambda(g,hkh^{-1})$ for 
all~$g$,~$h$,~$k\in G$. 

If $g\in G$, we let $G_g$ be the centralizer of~$g$ in~$G$. The map
$\chi_g:h\in G_g\mapsto\lambda(h,g)\in\CC^\times$ is a group morphism. The
group $G$ acts on $\hom(\Lambda^\bullet(V^g),S(V^g))$, so we may restrict
that action to~$G_c$ and consider the subspace of \newterm{semi-invariants}
$\hom(\Lambda^\bullet(V^g),S(V^g))^{G_g}_{\chi_g}$, that is, of the elements
$f\in\hom(\Lambda^\bullet(V^g),S(V^g))$ such that $h\cdot f=\chi_g(h)f$ for
all $h\in C_g$. Let now $\lin{G}$ be the set of conjugacy classes of~$G$
and for each $c\in\lin{G}$ let $g_c$ be a fixed element of~$c$. Then we have an
isomorphism $\Phi_c$
  \[
  \Phi_c:
  \hom(\Lambda^\bullet(V^{g_c}),S(V^{g_c}))^{G_{g_c}}_{\chi_{g_c}}[-d(g)]
    \to
      \left(\bigoplus_{g\in c}\hom(\Lambda^\bullet(V^g),S(V^g))\vee\omega_g\right)^G
  \]
such that for each
$f\in\hom(\Lambda^\bullet(V^{g_c}),S(V^{g_c}))^{G_{g_c}}_{\chi_{g_c}}$ we
have 
  \[
  \Phi_c(f) = \sum_{g\in G} \lambda(g,g_c)g(f)\vee\omega_{gg_cg^{-1}}.
  \]
It follows now that we have an isomorphism
  \[
  \bigoplus_{c\in\lin{G}}\Phi_c:
    \bigoplus_{c\in\lin{G}}\hom(\Lambda^\bullet(V^{g_c}),S(V^{g_c}))^{G_{g_c}}_{\chi_{g_c}}[-d(g)]
    \to
    \left(\bigoplus_{g\in G}\hom(\Lambda^\bullet(V^g),S(V^g))\vee\omega_g\right)^G,
  \]
and in this way we arrive at a well-known and very explicit description of
the Hochschild cohomology of the crossed product algebra
$S(V)\rtimes G$,
  \[
  \HH^\bullet(S(V)\rtimes G) =
         \bigoplus_{c\in\lin{G}}
                \hom(\Lambda^\bullet(V^{g_c}),S(V^{g_c}))^{G_{g_c}}_{\chi_{g_c}}[-d(g_c)]
  \]
originally obtained by Marco Farinati in~\cite{F} and Victor Ginzburg
and Dmitry Kaledin in~\cite{GK}; the paper \cite{SW} is a good reference
for this, too. In particular, if we let $\lin{G}_1$ be the
set of conjugacy classes $c\in\lin{G}$ such that $d(g_c)=1$, we have 
  \[ \label{eq:x:2}
  \HH^1(S(V)\rtimes G) = \hom(V,S(V))^G \oplus 
        \bigoplus_{c\in\lin{G}_1}S(V^{g_c})^{G_{g_c}}_{\chi_{g_c}}
  \]
and something nice happens: if $c\in\lin{G}_1$, then
$\chi_{g_c}(g_c)=\det\rho(g_c)\neq1$ and therefore
$S(V^{g_c})^{G_{g_c}}_{\chi_{g_c}}=0$, since $g_c\in G_{g_c}$. We thus see that, in fact, 
  \[
  \HH^1(S(V)\rtimes G) = \hom(V,S(V))^G.
  \]
Tracing back the isomorphisms involved, we can describe this
identification explicitly as follows.
If $r:V\to S(V)$ is a $G$-equivariant linear map, then one of the
universal properties of the symmetric algebra $S(V)$ implies that there is
a unique derivation $\bar r:S(V)\to S(V)$ which extends~$r$ and it turns
out to be $G$-equivariant. There is then a normalized derivation
$\delta_r:S(V)\rtimes G\to S(V)\rtimes G$ such that $\delta_r(fg)=
\bar r(f)g$ for all $f\in S(V)$ and $g\in G$. The class of this~$\delta_r$
is the element of~$\HH^1(S(V)\rtimes G)$ corresponding to the map~$r$.

\paragraph We are finally in position to describe the Lie module structure
of~$\HH^\bullet(S(V)\rtimes G)$ over the Lie algebra~$\HH^1(S(V)\rtimes
G)$ using our results from Section~\ref{sect:hh}.
We fix a $G$-equivariant map $r:V\to S(V)$ and let
$\delta=\delta_r:S(V)\rtimes G\to S(V)\rtimes G$ be the corresponding derivation 
described above. Let $T(V)$ be the tensor algebra on~$V$ and denote $\pi:T(V)\to
S(V)$ the natural surjection. Since $\pi$ is $G$-equivariant, it admits a
$G$-equivariant section $\sigma:S(V)\to T(V)$. On the other hand, and using
now a universal property of the tensor algebra, there
exists a unique linear derivation $D:T(V)\to S(V)\otimes V\otimes S(V)$ of
the algebra~$T(V)$ with values in $S(V)\otimes V\otimes S(V)$ endowed with
it obvious $T(V)$-bimodule structure such that $D(v)=1\otimes v\otimes 1$
for all $v\in V$, and it is $G$-equivariant. If $v\in V$, we will write
the element $D(\sigma(r(v)))$ of~$S(V)\otimes V\otimes S(V)$ in the form
$v\p1\otimes v\p2\otimes v\p3$ with an implicit sum, \`a~la~Sweedler. There
is a $\delta^e$-lifting $\delta_\bullet:K_\bullet\rtimes G\to
K_\bullet\rtimes G$ of~$\delta$ to the resolution $K_\bullet\rtimes G$ of
$S(V)\rtimes G$ such that
  \[
  \delta_p(1\otimes v_1\wedge\cdots\wedge v_p\otimes 1\rtimes 1)
    = \sum_{i=1}^p 
                v\p[i]1\otimes v_1\wedge\cdots\wedge
                 v\p[i]2\wedge \cdots\wedge v_p\otimes v\p[i]3\rtimes 1,
  \]
as one can check by direct computation. From this lifting we can construct
the map of complexes
  \[
  \delta_\bullet^\sharp:
        \hom_{(S(V)\rtimes G)^e}(K_\bullet\rtimes G,S(V)\rtimes G)
        \to 
        \hom_{(S(V)\rtimes G)^e}(K_\bullet\rtimes G,S(V)\rtimes G)
  \]
which up to natural isomorphisms in the lemma is identified with the map
  \[
  \delta_\bullet^\sharp:
        \hom(\Lambda^\bullet(V),S(V)\rtimes G)^G
        \to
        \hom(\Lambda^\bullet(V),S(V)\rtimes G)^G
  \]
given in each degree~$p\geq0$ by
  \[
  (\delta_p^\sharp\phi)(v_1\wedge\cdots\wedge v_p)
        = 
          \delta(\phi(v_1\wedge\cdots\wedge v_p))
          -
          \sum_{i=1}^p
                v\p[i]1
                \phi(v_1\wedge\cdots\wedge v\p[i]2\wedge\cdots\wedge v_p)
                v\p[i]3
  \]
for all $\phi\in\hom(\Lambda^p(V),S(V)\rtimes G)^G$. The right hand side in
this equation therefore a representative for the Gerstenhaber bracket
$[\delta,\phi]$.

\section{Tensor products, \texorpdfstring{$\Tor$}{Tor} and Hochschild homology}
\label{sect:tor}

\paragraph Let $A$ be an algebra and $\delta:A\to A$ a derivation, as
before.
Let $M$ and $N$ be a right and a left $A$-module, respectively, and let
$f:M\to M$ and $g:N\to N$ be $\delta$-operators on~$M$ and~$N$. There is a
linear map $f\boxtimes g:M\otimes_AN\to M\otimes_AN$ such that $(f\boxtimes
g)(m\otimes
n)=f(m)\otimes n+m\otimes g(n)$ for all $m\in M$ and all $n\in N$, and this
map depends naturally on the data used to construct it in the obvious sense.

If $\eta:Q_\bullet\to N$ is a projective resolution of~$N$ as a
left $A$-module and $g_\bullet:Q_\bullet\to Q_\bullet$ is a
$\delta$-lifting of~$g$ to~$Q_\bullet$, then we may consider the complex
$M\otimes_AQ_\bullet$ and the morphism $f\boxtimes
g_\bullet:M\otimes_AQ_\bullet\to M\otimes_AQ_\bullet$ which in each
homological degree~$i\geq0$ is given by~$f\boxtimes g_i$ and 
which induces upon passing to homology a map 
  \[
  H(f\boxtimes g_\bullet):
        H(M\otimes_AQ_\bullet)\to H(M\otimes_AQ_\bullet).
  \]There is an
analogue of Lemma~\pref{lemma:comm} and therefore proceeding as we did to
prove Theorem~\pref{thm:nabla} we obtain:

\begin{Theorem}\label{thm:tor}
If $M$ and $N$ are a right and a left $A$-module, respectively, and $f:M\to
M$ and $g:N\to N$ are $\delta$-operators, then there is a canonical morphism of
graded vector spaces
  \[
  \nabla^{f,g}_\bullet:\Tor^A_\bullet(M,N)\to\Tor^A_\bullet(M,N)
  \]
such that for each projective resolution $\eta:Q_\bullet\to N$ and each
$\delta$-lifting $g_\bullet:Q_\bullet\to Q_\bullet$ of~$g$ to~$Q_\bullet$
the diagram
  \[
  \begin{tikzcd}[column sep=1.5cm]
  H(M\otimes_AQ_\bullet) 
        \arrow[r, "\nabla^{f,g,Q_\bullet}_\bullet"] 
        \arrow[d, swap, "\cong"]
    & H(M\otimes_AQ_\bullet) 
        \arrow[d, "\cong"]
    \\
  \Tor^A_\bullet(M,N) 
        \arrow[r, "\nabla^{f,g}_\bullet"]
    & \Tor^A_\bullet(M,N)
  \end{tikzcd}
  \]
commutes.~\qed
\end{Theorem}

If in the situation of the theorem we also have a projective resolution
$\epsilon:P_\bullet\to M$ of $M$ and a $\delta$-lifting
$f_\bullet:P_\bullet\to P_\bullet$ of~$f$ to~$P_\bullet$, we may
consider the (total complex of the) tensor product
$P_\bullet\otimes_AQ_\bullet$ and on it the linear map $f_\bullet\boxtimes
g_\bullet:P_\bullet\otimes_AQ_\bullet\to P_\bullet\otimes_AQ_\bullet$
constructed in the obvious way. As the diagram
  \[
  \begin{tikzcd}[column sep=1.5cm]
  P_\bullet\otimes_AQ_\bullet
        \arrow[r, "f_\bullet\boxtimes g_\bullet"]
        \arrow[d, swap, "\epsilon\otimes\id_{Q_{\bullet}}"]
    & P_\bullet\otimes_AQ_\bullet
        \arrow[d, "\epsilon\otimes\id_{Q_{\bullet}}"]
    \\
  M\otimes_AQ_\bullet
        \arrow[r, "f\boxtimes g_\bullet"]
    & M\otimes_AQ_\bullet
  \end{tikzcd}
  \]
commutes, taking homology and observing that the morphism
$\epsilon\otimes\id_{Q_\bullet}$ induces the canonical isomorphism
$H(P_\bullet\otimes_AQ_\bullet)\cong H(M\otimes_AQ_\bullet)$, we find that
the diagram
  \[
  \begin{tikzcd}[column sep=1.5cm]
  H(P_\bullet\otimes_AQ_\bullet) 
        \arrow[r, "H(f_\bullet\boxtimes g_\bullet)"] 
        \arrow[d, swap, "\cong"]
    & H(P_\bullet\otimes_AQ_\bullet) 
        \arrow[d, "\cong"]
    \\
  \Tor^A_\bullet(M,N) 
        \arrow[r, "\nabla^{f,g}_\bullet"]
    & \Tor^A_\bullet(M,N)
  \end{tikzcd}
  \]
with vertical arrows the canonical isomorphisms, 
commutes and, proceeding symmetrically, that the same happens with
  \[
  \begin{tikzcd}[column sep=1.5cm]
  H(P_\bullet\otimes_AN) 
        \arrow[r, "H(f_\bullet\boxtimes g)"] 
        \arrow[d, swap, "\cong"]
    & H(P_\bullet\otimes_AN) 
        \arrow[d, "\cong"]
    \\
  \Tor^A_\bullet(M,N) 
        \arrow[r, "\nabla^{f,g}_\bullet"]
    & \Tor^A_\bullet(M,N)
  \end{tikzcd}
  \]
This means that the computation of the map $\nabla^{f,g}_\bullet$ is as
``balanced'' as that of $\Tor$ itself.

\paragraph If $\delta^e=\delta\otimes1+1\otimes\delta:A^e\to A^e$ is the
derivation of the enveloping algebra induced by~$\delta$ and if we view
$\delta:A\to A$ as a $\delta^e$-operator both on the left $A^e$-module~$A$, as
in~\pref{p:env}, and on the right $A^e$-module~$A$, and recalling
that in our situation we can identify the Hochschild homology
$\HH_\bullet(A)$ with $\Tor^{A^e}_\bullet(A,A)$, the construction
of Theorem~\pref{thm:tor} gives us a map
  \[
  \nabla_\bullet^{\delta,\delta}:\HH_\bullet(A)\to\HH_\bullet(A)
  \]
which we can compute in terms of any projective resolution of~$A$ as
an~$A^e$-bimodule, provided we can $\delta^e$-lift $\delta$ to it. If we
write what this amounts to when we use the $A^e$-projective resolution
$\epsilon:B(A)_\bullet\to A$ and the $\delta^e$-lifting described
in~\pref{p:env}, we find immediately that the map
$\nabla^{\delta,\delta}_\bullet$ induced on Hochschild homology by~$\delta$
coincides with the one considered by Tom Goodwillie in~\cite{Goodwillie}
or Jean-Louis Loday in Section~4.1 of~\cite{Loday}.


\end{document}